\documentclass[smallextended,numbook,runningheads]{svjour3}     
\smartqed  
\usepackage{graphicx}
\usepackage{upquote}
\usepackage{amsmath}
\usepackage{mathptmx}      
\newcommand{\C}{{\bf C}}
\newcommand{\Real}{\hbox{\kern .5pt \rm Re}\kern 1.3pt}
\newcommand{\Res}{\hbox{\kern .4pt\scriptsize\rm Re}\kern 1.1pt}
\newcommand{\ct}{\tilde c}
\newcommand{\et}{\tilde e}
\newcommand{\ft}{\tilde f}
\newcommand{\Et}{\tilde E}
\newcommand{\Dk}{D_k^{}}
\newcommand{\Dkr}{D_k^{(r)}}
\newcommand{\fk}{f_k^{}}
\newcommand{\xk}{x_k^{}}
\newcommand{\vark}{\varepsilon_k^{}}
\newcommand{\half}{{\textstyle{1\over 2}}}
\journalname{BIT}

\begin{document}

\title{Quantifying the ill-conditioning of analytic continuation}

\titlerunning{Analytic continuation} 

\author{Lloyd N. Trefethen}

\institute{Prof. L. N. Trefethen \at
Mathematical Institute\\
University of Oxford\\
Oxford, OX2 6GG, UK\\
\email{trefethen@maths.ox.ac.uk}}

\date{Received: date / Accepted: date}

\maketitle

\begin{abstract}
Analytic continuation is ill-posed, but becomes merely
ill-conditioned (although with an infinite condition number) if it is
known that the function in question is bounded in a given region
of the complex plane.  In an annulus, the Hadamard three-circles
theorem implies that the ill-conditioning is not too severe,
and we show how this explains the effectiveness of Chebfun and related
numerical methods in
evaluating analytic functions off the interval of definition.
By contrast, we show that analytic continuation is far more
ill-conditioned in a strip or a channel, with exponential loss of
digits of accuracy at the rate $\exp(-\pi x/2)$ as one moves along.
The classical Weierstrass chain-of-disks method loses digits at
the faster rate $\exp(-e\kern .3pt x)$.

\keywords{analytic continuation, Hadamard three-circles theorem, Chebfun}
\subclass{30B40}
\end{abstract}

\section{Introduction}
\label{intro}
Analytic continuation is well known to be ill-posed.  To be
precise, suppose a function $f$ is analytic in a connected open
region $\Omega$ of the complex plane and we know its values in a
set $E\subset \overline\Omega$ to an accuracy of $\varepsilon>0$.
(We assume $E$ is a bounded nonempty continuum whose closure
does not enclose any points of $\Omega\backslash \overline E$,
and that $f$ extends analytically to $E$.)  This implies no bounds
whatsoever on the value of $f$ at any point $z\in \Omega\backslash
\overline E$ (see Theorem~\ref{genthm1}).  And yet if we knew
$f$ {\em exactly\/} in~$E$, this would determine its values in
$\Omega$ exactly.

In practice, nevertheless, analytic continuation from inexact
data is carried out all the time, and what makes this possible
is {\em regularization,} the introduction of additional
smoothness assumptions.  Often an extrapolation technique is
applied without such assumptions being made explicit---and
this is understandable, for in applications, often one has a
sense of certain features of one's function without
being able to pin them down precisely.  In this paper, however,
we wish to be completely explicit and show how certain natural
regularizing assumptions lead to upper and lower bounds
on the accuracy of analytic continuation.

Our regularizing assumption will be that $f$ is not only analytic
in $\Omega$, but bounded.  For example, we can take
the bound to be
$\half$ and consider the set of functions
that are analytic in $\Omega$ and satisfy $\|f\|_\Omega^{} \le \half$.
(The symbol $\|\cdot\|_A^{}$ always denotes the supremum
norm over the set $A$.)
If $f,\ft$ are two such functions, then $\|\ft-f\|_\Omega^{} \le 1$.  We shall show 
(Theorem~\ref{genthm2}) that
if in addition $\|\ft-f\|_E^{}\le \varepsilon$, then for
each $z\in \Omega\backslash \overline E$,
\begin{equation}
|\ft(z) -f(z)| \le \varepsilon^{\alpha(z)}
\label{bound1}
\end{equation}
for some $\alpha(z)\in (0,1)$ that depends on $\Omega$, $E$,
and $z$ but not on $f$ and $\ft$ or $\varepsilon$.
Another way to say the same thing is
\begin{equation}
\log |\ft(z) -f(z)| \le \alpha(z) \log \varepsilon.
\label{bound2}
\end{equation}
We may interpret (\ref{bound2}) as follows: if we know a function
$f$ satisfying $\|f\|_\Omega\le \half$ to~$d$ digits on $E$,
then it is determined to $\alpha(z) \kern .3pt d$ digits at~$z$.
Even though our theorems are, of course, mathematical rather than
computational results, we shall use the terminology of digits a
good deal in discussing them, for this is an easy way to talk about
logarithmic quantities.

This general framework may sound rather abstract.  It becomes
concrete when we consider the dependence of $\alpha(z)$ on $z$
for particular choices of $\Omega$ and $E$, and after stating
a basic lemma in Section~2, we shall focus on two choices that
are particularly fundamental.  The first is radial geometry,
with analytic continuation outward from the unit disk $E$ into
a disk $\Omega$ of radius $R>1$ (Section~3).  In this setting
analytic continuation is reasonably well-conditioned, with
digits of accuracy being lost only linearly as $|z|$ increases.
This observation possibly goes back to Hadamard himself, and its
numerical implications have been considered by various authors
including Miller~\cite{miller} and Franklin~\cite{franklin}.
An intuitive way to understand the effect is to note that in the
limiting case $R\to\infty$, Liouville's theorem implies that $f$
must be constant, so if we know $f$ to accuracy $\varepsilon$
on~$E$, we know it to the same accuracy everywhere.  The result for
finite $R$ (Theorem~\ref{thm1}) can be derived from the Hadamard
three-circles theorem~\cite{hille}.  The essence of the matter
is that analyticity and boundedness at a large radius imply rapid
exponential decrease of Taylor coefficients, hence good behavior
at smaller radii.

Analytic continuation is much more difficult in the other geometry
we focus on, which is linear (Section~4).  Here we take $\Omega$
to be an infinite half-strip of half-width~$1$ (without loss
of generality), and $E$ as the end segment of the half-strip.  As $z$
moves away from $E$ along the centerline of the strip, digits are
lost exponentially as a function of distance, and we prove this by
reducing the problem to the configuration of Lemma~\ref{thelemma}.
At a point $z$ that is $2\pi$ units away from the end, for
example, the number of accurate digits has shrunk by a factor
$(\pi/4)\exp(\pi^2)\approx \hbox{15,000}$, so if you want to
have $3$ digits of accuracy at such a point, you'll need to
start with 45,000 digits.  

Our formulations are conformally invariant, and thus different
regions $\Omega\ne \C$ can be transplanted from one to another.
In particular, analytic continuation in a disk and a half-strip
are essentially equivalent problems,
and the reason the half-strip is exponentially more difficult
than the disk is that the conformal map that relates them is an
exponential.  Section 5 explores results for general regions (not
necessarily simply connected) that follow from these observations,
presenting theorems establishing the behavior asserted in the
opening paragraphs of this introduction.

Along the way, we shall relate our results to numerical algorithms.
Section 3 presents a simple method for numerical analytic
continuation in a disk that approximately achieves the bounds
indicated in Theorem~\ref{thm1}, based on Taylor series on the
disk, and in Section 6 we show that this method is implicit
in Chebfun~\cite{chebfun,atap}.  An algorithm of this kind was
proposed by Franklin~\cite{franklin}, and there is
recent related work by Demanet and coauthors~\cite{batenkov,demtow},
among others.  Further numerical algorithms and associated
mathematical estimates for analytic continuation can be found
in~\cite{cannm,douglas,fdfd,fdfq,fzcm,hen66,henrici,miller,niet,reichel,stef,vessella}.
More generally, there is a large literature of numerical methods for
ill-posed problems, which are often defined by partial differential
or integral equations.  One paper that speaks of the connection
between analytic continuation and more general ill-posed problems
defined by PDE\kern .3pt s is~\cite{miller}.

In Section 7 we turn to the most famous algorithm of analytic
continuation, which goes back to Weierstrass:
marching Taylor expansions from one overlapping disk to another
in a chain.  We show that this method, if carried out
numerically in the half-strip with a certain optimal choice of
parameters, suffers exponential loss of accuracy at a rate $2\kern
.3pt e/\pi$ times faster than the optimal rate in a half-strip,
so that if one marches $2\pi$ units down the half-strip, the
number of accurate digits is divided by $\exp(2\pi e) \approx
\hbox{26,000,000}$. 

Before turning to the details, we comment on the relationship
between approximation methods based on multiple derivative values
at a single point, such as chain-of-disks continuation of Taylor
series or Pad\'e approximation~\cite{bgm}, and methods based just
on function values but at multiple points.  Our formulations are
of the latter form, but the two contexts are close.  Thanks to
the standard lemma of complex analysis known as Cauchy's estimate,
knowing a function $f$ on the unit disk to accuracy~$\varepsilon$
is approximately the same as knowing its Taylor coefficients $c_k$
to accuracy~$\varepsilon$, and more generally, if $f$ is known to
accuracy~$\varepsilon$ on the closed disk of radius $r$, this is
approximately the same as knowing its Taylor coefficients $c_k$
to accuracy~$\varepsilon \kern .5pt r^{-k}$.

\section{A lemma}
Our results are based on the following lemma,
the {\em Hadamard three-lines theorem,}
a more general form of which can be found in~\cite[Thm.~12.8]{rudin}.
Numerical algorithms for this geometry are discussed in~\cite{fdfq}.

\begin{lemma}
\label{thelemma}
Let\/ $h$ be an analytic function in the infinite strip
$S = \{ w: \, 0< \Real w < 1\}$
with $\|h\|_S^{}\le 1$ and\/ $\lim_{u\downarrow 0}^{} 
\sup_v|h(u+iv)| \le \varepsilon$ for some\/ $\varepsilon \in(0,1)$.  Then
for all $w\in S$,
\begin{equation}
\log|h(w)| \le (1-\Real w)\log\varepsilon,
\hbox{\quad i.e.,\quad}
|h(w)| \le \varepsilon^{1-\Res w}.
\label{lemmaeq}
\end{equation}
Conversely, for any\/ $\varepsilon\in(0,1)$, there is a function\/ $h$
satisfying the given conditions for which the inequalities
$(\ref{lemmaeq})$ hold as equalities for all\/ $z$ with $0 < \Real z
< 1$.
\end{lemma}

\begin{proof}
We begin by noting that without loss of generality, we
may suppose that $h$ is analytic on the closed set
$\overline S$.  If not, we could restrict attention to a smaller
domain $\delta\le \Real w \le 1-\delta$ for $\delta<0$ and then
take the limit $\delta \to 0$.

Set $\nu = -\log\varepsilon$, implying $e^{-\nu} = \varepsilon$.
The function $e^{-\nu \kern .5pt w} h(w)$ is analytic in $\overline
S$ and bounded in absolute value by $\varepsilon$ for $\Real w =
0$ and also for $\Real w = 1$.  Therefore, by the maximum modulus
principle as qualified in the next paragraph, it is bounded
by~$\varepsilon$ for all $w\in \overline S$.  Thus $|h(w)| \le
\varepsilon\kern .5pt e^{\kern .7pt \nu \Res w} = e^{(1-\Res
w)\log \varepsilon}$, as required.

The qualification just mentioned is that the maximum modulus
principle does not apply to arbitrary functions on an unbounded
domain with a gap in the boundary at~$\infty$.  However,
this function is known to be bounded in an infinite strip,
and in such a situation, according to a Phragm\'en--Lindel\"of
theorem~\cite[Thm.~18.1.4]{hille}, the maximum modulus principle
applies after all.

For the converse, it is enough to consider the function
$h(w) = \varepsilon \kern .5pt e^{\kern .7pt \nu w}$.
\qed
\end{proof}

\section{Analytic continuation in a disk}
Figure~\ref{geom1} shows our first fundamental geometry, one that
has been considered by a number of authors.  We take $E$ to be the
closed unit disk and $\Omega$ as the open disk of radius $R>1$.
As always, our concern is obtaining bounds on $|\kern .3pt
g(z)|,$ $z \in \Omega\backslash E$, for an analytic function $g$
($= \ft - f\kern .7pt$) satisfying $\|\kern .3pt g\|_\Omega^{}
\le 1$ and $\|\kern .3pt g\|_E^{}\le \varepsilon$.  Here is the
result, essentially the {\em Hadamard three-circles theorem,} with
$\|\cdot\|_r^{}$ denoting the supremum norm over $\{z:\, |z|< r\}$.

\begin{figure}
\begin{center}
\vskip .6in
\includegraphics[scale=.6]{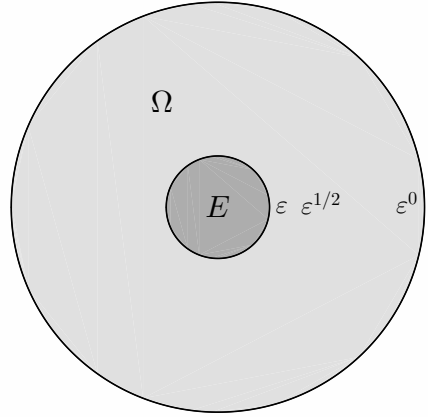}
\end{center}
\caption{\label{geom1}Analytic continuation of a function
$f$ from the unit disk\/ $E$ to a larger disk\/~$\Omega$ where it
is bounded.
If $f$ is known to $d$ digits on $E$, the number of digits 
determined in\/ $\Omega$ falls off smoothly to $0$ at the outer boundary.}
\end{figure}

\begin{theorem}
Given $R>1$, let $g$ be analytic in $\Omega = \{z\in\C: \,
|z|< R\}$ with $\|\kern
.3pt g\|_R^{}\le 1$ and
$\|\kern .3pt g\|_1^{}\le \varepsilon \in (0,1)$.  Then for
any\/ $z$ with $1< |z|< R$,
\begin{equation}
\log |\kern .3pt g(z)| \le \alpha(z) \log\varepsilon, \hbox{\quad i.e.,\quad}
|\kern .3pt g(z)| \le \varepsilon^{\alpha(z)},
\label{ineq1}
\end{equation}
where
\begin{equation}
\alpha(z) = 1 - {\log |z|\over \log R}.
\label{alfdef}
\end{equation}
Conversely, for an infinite sequence of 
values\/ $\varepsilon$ converging to\/ $0$, there are functions~$g$ satisfying
the given conditions for which the inequalities
$(\ref{ineq1})$ hold as equalities for all\/~$z$ with $1<|z|<R$.
\label{thm1}
\end{theorem}

\begin{proof}
As in the proof of Lemma~\ref{thelemma}, we begin by noting that
without loss of generality, we may suppose that $g$ is analytic in
the closed domain $|z| \le R$.  We transplant the problem to
the infinite strip $\overline S$ of the last section by defining
$h(w) = g(z)$ for $z = R^{\kern .8pt w} = e^{w\log R}$, hence $
w = \log z / \log R$ (it doesn't matter which branch of $\log z$ is
used, as they all lead to the same estimate).
By the lemma, we have $\log|\kern
.3pt g(z)|\le (1- \Real w)\log \varepsilon = (1-\log|z|/\log
R)\log\varepsilon$, as claimed.

For the converse result, consider $g(z) = (z/R)^n$ with
$n = -\log\varepsilon / \log R$.  For an infinite sequence of values
$\varepsilon\to 0$, $n$ is an integer, and in these cases $g$ is an
analytic function with the required properties.
\qed
\end{proof}

In words, we can describe Theorem~\ref{thm1} as follows.  If $f$
is known to $d$ digits for $|z|\le 1$, the number of digits
determined for $|z|=r$ diminishes to $0$ as $r\to R\kern
1pt$; as a function of $\log r$, the loss of digits is linear.
At $r=\sqrt R$, for example, $f$ is determined to $d/2$ digits.

It is easy to outline an algorithm for analytic continuation,
based on a finite Taylor series of length $n \approx -\log\varepsilon/\log R$,
that achieves approximately the accuracy promised in
Theorem~\ref{thm1}.
From approximate values 
for $|z|=1$ with error at most $\varepsilon$ of a function~$f$
with $\|f\|_\Omega^{}\le \half$,
we compute approximations $\ct_k\approx c_k$ to the 
Taylor coefficients $\{c_k\}$ of $f$ for $0\le k \le n$
with $|\ct_k-c_k|\le \varepsilon$.
That this is possible follows from Cauchy's estimate applied
on the circle $|z|=1$; in practice,
we sample $f$ on a grid of $N\gg n$ roots of unity and
use the Fast Fourier Transform~\cite{akt}.  By Cauchy's
estimate applied now for $|z|=R$, the Taylor coefficients
of $f$ satisfy $|c_k| \le \half R^{-k}$.
If we define
\begin{equation}
\ft(z) = \sum_{k=0}^n \ct_k z^k,
\label{ftildeseries}
\end{equation}
then we have
$$
\ft(z)-f(z) = \sum_{k=0}^n (\ct_k-c_k) z^k - \sum_{k=n+1}^\infty
c_kz^k,
$$
implying
$$
|\ft(z)-f(z)| \le \sum_{k=0}^n \varepsilon |z|^k + \half\sum_{k=n+1}^\infty
(|z|/R)^k.
$$
Our choice $n \approx -\log\varepsilon/\log R$ implies 
$\varepsilon\approx R^{-n}$, and thus
these two sums are both of size on the order of $(|z|/R)^n$.
We therefore achieve, as required,
$$
|\ft(z)-f(z)| \approx (R/|z|)^{-n} = R^{-\alpha(z)n} \approx
\varepsilon^{\alpha(z)},
$$
with the equality in the middle holding
since the definition (\ref{alfdef}) implies $R^{\kern .8pt\alpha(z)}
= R/|z|$.

In the algorithm just described, the regularization occurred
when we took the series (\ref{ftildeseries}) to be finite rather
than infinite.
In this geometry, the ill-posedness of analytic
continuation resides in the fact that as $k\to\infty$, powers $z^k$
have unbounded discrepancies of absolute value between one radius
$|z|$ and another.  For a fascinating analysis of the implications
of such behavior in the context of computation of
Taylor coefficients, see~\cite{born}.
The importance of truncating a series for numerical analytic
continuation was
recognized at least as early as~\cite{lewis}, and
a detailed error analysis of an algorithm with this flavor can
be found in~\cite{franklin}.

We have described the algorithm as a process for working with
a function known to accuracy $\varepsilon$ on the unit disk.
One way to obtain such data is to take Taylor polynomials of $f$
of successively higher degrees (in exact arithmetic), in which
case (\ref{ineq1}) amounts to the statement that the convergence
of Taylor polynomials to $f(z)$ for $|z|<R$ is exponential at a
rate of order $(|z|/R)^n$.

\section{Analytic continuation in a half-strip}
Our second fundamental geometry is shown in Figure~\ref{geom2}.
Now $E$ is the complex interval $(-i,i\kern .7pt )$ and $\Omega$ is the
half-strip~$H$ of points $z=x+iy$ with $x> 0$, $-1<y<1$.
Our aim is to analytically continue a function $f$ from the end
segment of the strip to real positive values $x$.

\begin{figure}
\begin{center}
\vskip .8in
\includegraphics[scale=.6]{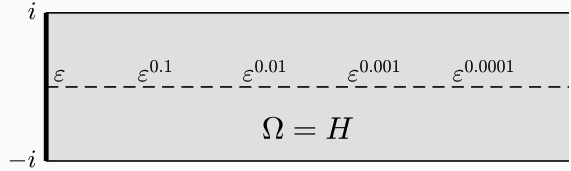}
\end{center}
\caption{\label{geom2}Analytic continuation of a bounded function
$f$ along the centerline of an infinite half-strip $H$ of
half-width $1$. The number of digits
determined falls off exponentially with distance from the data,
reducing by a factor of\/ $10$ each time\/ $x$ increases by
$(2/\pi)\log 10 \approx 1.47$.}
\end{figure}

To apply Lemma~\ref{thelemma}
in this geometry, we need to map $H$ conformally to the infinite
strip $S$ of Section 2 in such a way that the corner points
$z = \pm \kern .3pt i$ map to the infinite vertices $w = \pm \kern
.3pt i\kern .5pt
\infty$ and
$z=\infty$ maps to $w=1$.
We can construct such a map by composing three simpler maps.
First, $u = \sinh(\pi z/2)$ maps $H$ to the right half-plane with
distinguished points $\pm i$ and~$\infty$.
Next, $v = {(i-u)/(i+u)}$ maps the right half-plane to the
upper half-plane with distinguished points $0$, $\infty$, and $-1$.
Finally, $w = (-i/\pi)\log v$ maps the upper half-plane to
$S$ as required.
Combining these steps, we find that
the map from the half-strip $H$ in the $z$-plane to the
infinite strip $S$ in the $w$-plane is given by
\begin{equation}
\label{map}
w = -{i\over \pi} \log\left( {i - \sinh(\pi z/2)\over i + \sinh(\pi
z/2)} \right) =
{2\kern .3pt i\over \pi}\log\left( {1-i\kern .5pt\exp(\pi z/2)\over
-i + \exp(\pi z/2)} \right).
\end{equation}

The properties of this map are such as to reveal that analytic
continuation into a half-strip where a function is known to
be bounded, while possible in principle, is so ill-conditioned as to 
be generally infeasible in practice.  Instead of digits
of accuracy being lost linearly, they are lost exponentially,
as emphasized in Figure~\ref{geom2}.
To see how this comes about,
consider the situation in which $z$ is a real number
$x>0$.  The leading
terms in the asymptotics for $u$, $v$, and $w$ for large $x$
give us
\begin{equation}
u \sim \half e^{\pi x / 2}, \quad
v \sim -1 + 4\kern .3pt i\kern .5pt e^{-\pi x/2}, \quad
w \sim 1 - {4\over \pi} \kern .7pt e^{-\pi x/2}.
\label{asymp1}
\end{equation}
Since $w$ is exponentially close to $1$, Lemma~\ref{thelemma}
implies that the number of digits of accuracy will
be multiplied by an exponentially small factor
${\sim}\kern 1pt (4/\pi) \exp(-\pi x /2)$. 
A more careful analysis sharpens (\ref{asymp1}) to
\begin{equation}
{4\over \pi} \kern .7pt e^{-\pi x/2} - \left({4\over \pi} - 1\right)
e^{-\pi x} \,\le\,
1- w \,\le\,  {4\over \pi} \kern .7pt e^{-\pi x/2}
\quad (x \ge 0),
\label{asymp2}
\end{equation}
from which we get the following theorem:

\begin{theorem}
Let $g$ be analytic in the half-strip $H$ with $\|\kern .3pt
g\|_H^{}\le 1$ and
$\|\kern .3pt g\|_E^{}\le \varepsilon$ 
for some\/ $\varepsilon \in(0,1)$.  Then for
any\/ $x > 0$,
\begin{equation}
\log |\kern .3pt g(x)| \le \alpha(x) \log\varepsilon, \hbox{\quad i.e.,\quad}
|\kern .3pt g(x)| \le \varepsilon^{\alpha(x)},
\label{ineq2}
\end{equation}
where
\begin{equation}
\alpha(x) = {4\over \pi} \kern .7pt e^{-\pi x/2}
- \left({4\over \pi} - 1\right) e^{-\pi x}.
\label{alfdef2}
\end{equation}
Conversely, for any $\varepsilon\in (0,1)$, there is a function
$g$ satisfying the given conditions for which,
for all $x\ge 0$,
\begin{equation}
\log |\kern .3pt g(x)| \ge \beta(x) \log\varepsilon, \hbox{\quad i.e.,\quad}
|\kern .3pt g(x)| \ge \varepsilon^{\beta(x)},
\label{conv}
\end{equation}
with
\begin{equation}
\beta(x) = {4\over \pi} e^{-\pi x/2}.
\label{betadef}
\end{equation}
\label{thm2}
\end{theorem}

\begin{proof}
This follows from Lemma~\ref{thelemma} by the conformal
transplantation (\ref{map}), using the bounds (\ref{asymp2}).
\qed
\end{proof}

\section{General geometries}
For more general geometries than the disk or the half-strip,
let us now justify the claims made in the introduction.  
First is the ill-posedness statement of the opening paragraph, which
as usual we formulate as an assertion about an analytic
function $g = \ft - f$.

\begin{theorem}
\label{genthm1}
Let $\Omega$ be a connected open region of the complex plane
$\C$ and let~$E$ be a bounded nonempty continuum in $\overline\Omega$
whose closure $\overline E$ does not enclose any points of\/ $\Omega\backslash
\overline E$.
Let $g$ be an analytic function in\/ $\Omega\cup E$ satisfying
$\|\kern .3pt g\|_E^{} \le \varepsilon$ for some $\varepsilon> 0$.
This condition implies
no bounds whatsoever on the value of\/~$g$ at any point $z\in
\Omega\backslash \overline E$.  
\end{theorem}

\begin{proof}
Given $z\in \Omega\backslash\overline E$ and a complex
number $M$, we shall show there is a polynomial $p$ such
that $\|\kern .5pt p\|_E^{}\le\varepsilon$ and $p(z) = M$.
Let $\Et$ denote the
compact set consisting of $\overline E$ together
with all points enclosed by this set; thus the complement of
$\Et$ in the complex plane is connected.
By assumption, $z \not\in \Et$.
According to Runge's theorem~\cite[Thm.~13.7]{rudin},
there is a 
polynomial $q$ such that $\|\kern .3pt q\|_{\Et}^{} \le \varepsilon/2$ 
and $|\kern .3pt q(z)-M| \le \varepsilon/2$.
Now define $p(\zeta) = q(\zeta) + M - q(z)$.
We readily verify $p(z) = M$ and $|\kern .5pt p(\zeta)|
\le |\kern .3pt q(\zeta)| + |M-q(z)| \le
\varepsilon$ for all $\zeta\in E$.
\qed
\end{proof}

The well-posedness statement of the introduction is (\ref{bound1}),
which we formulate as follows.

\begin{theorem}
\label{genthm2}
Let $\Omega$, $E$, $\varepsilon$, and $g$ be as in
Theorem~$\ref{genthm1}$,
but now with $g$ additionally satisfying $\|g\|_\Omega^{}\le 1$,
and let\/ $z$ be a point in $\Omega\backslash \overline E$.
Assume that the boundary of $E$ is piecewise smooth (a finite
union of smooth Jordan arcs).
Then there is a number\/ $\alpha\in (0,1)$, independent
of\/ $g$ though not of~$z$, such that for all\/ $\varepsilon>0$, 
\begin{equation}
|\kern .3pt g(z)| \le \varepsilon^{\alpha(z)}.
\label{genthm2eq}
\end{equation}
\end{theorem}

\begin{proof}
Let $\Gamma$ be a smooth open arc in $\Omega$ connecting
a point $e$ in the boundary of $E$, which will
necessarily belong to $\overline \Omega$, to $z$.  For a sufficiently
small $\delta>0$, the intersection of the open $\delta$-neighborhood
of $\Gamma$ with $\Omega$ is a simply-connected region $J$
in $\Omega$ with $z$ in its interior and a piecewise smooth
boundary, 
a portion of $E$ making up part of this boundary.  By a conformal
map, we may transplant $J$ to the domain of Figure~\ref{geom2},
whereupon Theorem~\ref{thm2} provides a suitable (if
typically very pessimistic) value of $\alpha$.
\qed
\end{proof}

Theorem~\ref{genthm2} asserts
that analytic continuation in the presence of a boundedness
condition is a well-posed problem in the sense
that there is a unique solution depending continuously
on the data, but this does not mean that its condition
number is finite.  A finite condition number would correspond
to $|\kern .3pt g(z)|$ shrinking linearly with~$\varepsilon$, that is,
to a value $\alpha = 1$ in (\ref{genthm2eq}), or more
generally to the bound
\begin{equation}
|g(z)| \le C\kern .3pt\varepsilon
\label{moregen}
\end{equation}
as $\varepsilon\to 0$ for some constant $C$ depending on $z$ but
not $g$.
But Theorem~\ref{genthm2} only gives
$\alpha < 1$, and in fact, we now show that $\alpha = 1$ cannot occur except
in trivial cases with $\overline \Omega = \C$.

\begin{theorem}
\label{genthm3}
Under the circumstances of Theorem~$\ref{genthm2}$, 
a bound of the form $(\ref{moregen})$ can never hold unless
$\overline\Omega$ is the whole complex plane $\C$.
\end{theorem}
\begin{proof}
If $\overline\Omega$ is all of $\C$, we may be in the trivial
situation mentioned in the introduction, where Liouville's theorem
implies that $g$ is constant.  (This will be true, for example,
if $\Omega$ consists of $\C$ with a finite set of points removed.
It won't be true if~$\Omega$ consists of $\C$ with some arcs
removed.)
On the other hand suppose there
is a point $z_0^{}\in\C$ disjoint from $\overline\Omega$.  Then
there is a closed disk $\Delta$ about~$z_0^{}$ that is disjoint
from $\overline\Omega$.
By the conformal map $1/(z-z_0^{})$, we may transplant the problem
so that $\Omega$ and $E$ are bounded.
For any $z \in \Omega\backslash\overline E$,
as in the proof of Theorem~\ref{genthm1}, Runge's theorem
ensures that there is a polynomial $p$ such that
$p(z) = 0$ and $\Real p(\zeta) \le -1$ for all $\zeta\in E$.
Choose $M>0$ such that $\Real p(\zeta) \le M$ for all $\zeta\in \Omega$.
Then $q(z) = p(z)-M$ has real parts $\le -(1+M)$ for $\zeta\in E$,
$-M$ at $z$, and $\le 0$ for $\zeta\in \Omega$.
Given $\varepsilon\in (0,1)$, define
$g(\zeta) = \exp[\kern 1pt\log(\varepsilon)(-q(\zeta))/(1+M)]$.
Then $\|\kern .3pt g\|_\Omega^{}\le 1$
and $\|\kern .3pt g\|_E^{} \le \varepsilon$, but
$|\kern .3pt g(z)| = \varepsilon^{M/(M+1)}$.
This contradicts (\ref{moregen}).
\qed
\end{proof}

Together, Theorems~\ref{genthm2} and \ref{genthm3} assert that
analytic continuation with a boundedness condition in a nontrivial
region is always well-posed but always has an infinite condition
number.  I am not aware if such a general assertion has been
made before.

\section{Analytic continuation in Chebfun}
This project sprang from work with Chebfun, a software system for
numerical computing with functions~\cite{chebfun}.  In its basic
mode of operation, Chebfun works with smooth functions
on an interval that without loss of generality we may take to
be $[-1,1]$.  Chebfun represents each function to ${\approx}\kern
1pt 16$ digit precision by a polynomial in the form of a finite
Chebyshev series, and early in the project, it was realized that
the same series could be used for evaluation at complex points
off the interval.

\begin{figure}
\begin{center}
\smallskip
\vskip .2in
\includegraphics[scale=.7]{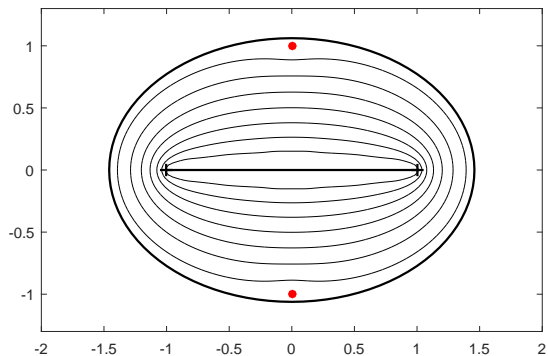}
\end{center}
\caption{\label{chebfig} Chebfun implicitly carries out analytic
continuation within a Bernstein ellipse bounded approximately
by the nearest complex
singularity of a function\/ $f$ defined on a real interval.
For the function $f(x) = \log(1+x^2)$ defined on $[-1,1]$,
the thick outer curve shows the ``Chebfun ellipse'' estimate of the
region of analyticity plotted by \kern 1pt{\tt plotregion} and the
inner curves are contour lines corresponding to errors
$|f(x) - p(x)| = 10^{-2}, 10^{-4}, \dots, 10^{-14}$ (from outside
in), where $p$ is the Chebfun approximation, a polynomial of
degree $38$.  The dots mark the actual branch points of~$f$.}
\end{figure}

Figure~\ref{chebfig} illustrates this effect for the function $f(x)
= \log(1+x^2)$, which has branch points at $x = \pm\kern .5pt  i$.
By an adaptive process described in~\cite{chopping}, the Chebfun
command \verb|p = chebfun('log(1+x^2)')| constructs a polynomial
$p$ that matches $f$ on $[-1,1]$ with a maximal error of about
$2^{-51} \approx 4.44\times 10^{-16}$; the degree of $p$ for this
example is 38.\ \ Typing {\tt p(0)} to evaluate the polynomial
at $x=0$, for example, returns the value $5.5\times 10^{-17}$,
accurate to more than 16 digits.  What is interesting is that
typing {\tt p(i/2)} also gives an accurate value: $-0.2876820630$,
as compared with the true value $\log(0.75) \approx -0.2876820768$.
How can we explain this?

In fact, Chebfun is carrying out the Chebyshev analogue of the
algorithm of Section~3: it computes a finite sequence of Chebyshev
series coefficients, then uses these coefficients to define an
approximation $p = \ft$.  Instead of working outward from the
unit disk $E$ to larger disks, it is working outward from the
unit interval $E$ to so-called Bernstein ellipses, whose algebra
is defined by a transplantation of the results of Section~3 by
the Joukowski map $(z+z^{-1})/2$.  For details, see~\cite{atap},
particularly the discussions of the ``Chebfun ellipse'' and the
command {\tt plotregion}.
According to Theorem~\ref{thm1} as transplanted from disks to
ellipses, we can expect the number of accurate digits to fall
off smoothly, and Figure~\ref{chebfig} shows that this is just
what is observed.

Analytic continuation in a region bounded by an ellipse is
mentioned as Example~3 of~\cite{franklin}, and a detailed analysis
of algorithms in this geometry is presented in~\cite{demtow}.
Results for ellipses analogous to those of~\cite{born} for disks
can be found in~\cite{wang}.

\section{Analytic continuation by a chain of disks}

The classic idea for analytic continuation, going back to
Weierstrass, involves a succession of Taylor expansions,
each with its own disk of convergence.  In principle, this
procedure enables one to track a function along any path where
it is analytic, and we recommend the beautifully illustrated
discussion in Section~3.6 of Wegert's {\em Visual Complex
Functions}~\cite{wegert}.  For inexact function data, however,
the method is far from promising.  There is a small literature
on numerical realizations, and a memorable contribution is a 1966
paper by Henrici in which the necessary transformations of series
coefficients are formulated in terms of matrix multiplications;
see~\cite{hen66} or~\cite[sec.~3.6]{henrici}.

To analyze this idea quantitatively, the simplest setting is
a channel, essentially the same as the
half-strip of Section~4.  Specifically,
consider the finite-length ``stadium'' $G$ shown in 
Figure~\ref{stadium}.  Given $L>0$, this is the strip
of half-width $1$ extending from $x=0$ to $x=L$ together with
half-disks of radius~1 at each end.  For some $n>0$ and
$r\in (0,1),$ we define $h=L/n$ and
\begin{equation}
\xk = kh, ~~\Dk = \{z: \, |z-x_k|< 1\}, ~~
\Dkr = \{z: \, |z-x_k|< r\}
\label{kdefs}
\end{equation}
for $0 \le k \le n$.
We assume $n$ is large enough so that $h < 1-r$.

\begin{theorem}
With the definitions of the last paragraph,
let $f$ be analytic in\/ $G$ with $\|f\|_G^{}\le 1$
and let $\fk$ be analytic in $D_k^{}$
with $\|\fk\|_{D_k^{}}^{}\le 1$, $0\le k \le n-1$.
Assume $r\le 1/2$ and $h\le 1/4$.
Given $\varepsilon\in (0,1)$, define
\begin{equation}
\vark = \varepsilon^{\kern .5pt \exp(-\et\kern .8pt \xk)}, \quad
\et = {1+2h\over r\log(1/r)}.
\label{vardef}
\end{equation}
If
\begin{equation}
\|f_0^{}-f\|_{D_0^{(r)}} \le \varepsilon_0^{}
\label{basecase}
\end{equation}
and
\begin{equation}
\|\fk-f_{k-1}^{}\|_{D_k^{(r)}} \le \vark,  \quad 1\le k \le n-1,
\label{ineqs}
\end{equation}
then for all sufficiently small choices of $\varepsilon$, 
\begin{equation}
\|f_{k-1}^{}- f\|_{D_k^{(r)}}^{} \le \vark, \quad 1 \le k \le n.
\label{induc2}
\end{equation}
\label{chainthm}
\end{theorem}

\begin{figure}
\begin{center}
\vskip .25in
\includegraphics[scale=.82]{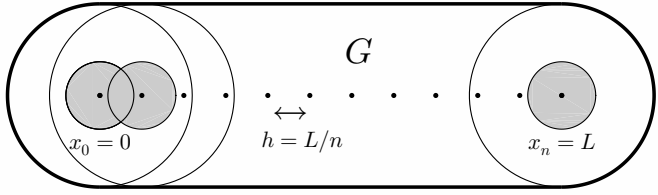}
\end{center}
\caption{\label{stadium} Chain-of-disks analytic continuation of a
function $f$ along a channel $G$ of half-width~$1$
where it is assumed to be bounded
and analytic.  One starts with $f$ known to accuracy
$\varepsilon_0^{}
=\varepsilon$ in the shaded disk \smash{$D_0^{(r)}$} of radius
$r$ about $x_0 = 0$; this is then used
to construct an expansion to accuracy~$\varepsilon_1^{}$ in the next
shaded disk about $x_1^{}=h$, and so on.  The 
best accuracy is achieved with $r = 1/e$ and
$h\to 0$, with the number of accurate digits diminishing at
the rate $\exp(-e\kern .3pt x)$.}
\end{figure}

\begin{proof}
We proceed by induction on $k$ in (\ref{induc2}).
The case $k=1$ follows from (\ref{basecase}) by
Theorem~\ref{thm1} (rescaled by a factor $R=1/r$).
Consider step $k+1$, assuming (\ref{induc2}) has been established
for previous steps.
Combining (\ref{ineqs}) and (\ref{induc2}) gives
\begin{equation}
\|\fk - f\|_{\Dkr}^{} \le 2\kern .5pt \vark.
\label{induc4}
\end{equation}
By Theorem~\ref{thm1} (with the same rescaling as
before), (\ref{induc4}) implies
\begin{equation}
\|\fk - f\|_{D_{k+1}^{(r)}}^{} \le 
(2\kern .5pt \vark)^\alpha, \quad \alpha = 
1 - {\log(1+h/r)\over \log (1/r)}. 
\label{induc3}
\end{equation}
(We avoid replacing $\log(1/r)$ by $-\log r$ since it can be
confusing to have to remember that $\log r$ is negative.)
We are done if we can show
\begin{displaymath}
(2\kern .5pt \vark)^\alpha \le \varepsilon_{k+1}^{}.
\end{displaymath}
or by taking logarithms,
\begin{displaymath}
\left( 1 - {\log(1+h/r)\over \log (1/r)}\right)
(\log 2 + \log(\vark)) \le
\log(\varepsilon_{k+1}^{}).
\end{displaymath}
By (\ref{vardef}), if we divide both sides by the negative
quantity $\log(\vark)$, this becomes
\begin{equation}
\left( 1 - {\log(1+h/r)\over \log (1/r)}\right)
\left({\log 2 \over \log(\vark)} + 1\right) \ge
\exp\left({-h-2h^2\over r \log(1/r)}\right) .
\label{toprove}
\end{equation}
Now suppose for a moment that $\vark$ is negligible.  Then
the $\log 2 /\log(\vark)$ term goes away and
the condition we must verify reduces to
\begin{displaymath}
1 - {\log(1+h/r)\over \log (1/r)}
\ge \exp\left({-h-2h^2\over r \log(1/r)}\right) .
\end{displaymath}
A numerical search readily confirms that this holds
with a strict
inequality over the indicated region $r\in(0,1/2)$, $h\in (0,1/4)$
(the coefficient of the term $-2\kern .3pt h^2$ was introduced
to ensure this).
Because of the assumption in the theorem statement that~$\varepsilon$ is
sufficiently small, this establishes (\ref{toprove}).
\qed
\end{proof}

The conclusion of
Theorem~\ref{chainthm} becomes memorable
in the limit $\varepsilon,h\to 0$.  
The parameter $\et$ of (\ref{vardef}) is
then minimized with the choice $r=1/e$, for which
it takes the value $\et = e$.
We conclude that with this optimal choice of $r$,
\begin{displaymath}
\hbox{\em The number of accurate digits in chain-of-disks continuation}
\end{displaymath}

\vskip -28pt 

\begin{displaymath}
\hbox{\em along a channel of half-width\/ $1$ 
decays at the rate $\exp(-e\kern .3pt x)$.}
\end{displaymath}

In a field as established as complex analysis, it is hard to be
sure that anything is
entirely new, but I am not aware of a previous estimation of
loss of digits at the rate $\exp(-e \kern .3pt x)$ for
chain-of-disks continuation.  The
general observation of exponential loss of information is more
than a century old.   Henrici~\cite{hen66} writes (his italics)
\begin{displaymath}
\hbox{\em The early vectors$\dots$ must be computed more accurately than 
the late ones}
\end{displaymath}
and he gives credit for related work to Mittag-Leffler,
Painlev\'e, Zeller, and Lewis~\cite{lewis}.

It is interesting to note what our estimates suggest for what
might be considered a very natural test problem for analytic continuation.
Suppose we have a function like $f(z) = \sqrt z$ 
that is known to be known to be bounded and analytically continuable
along any curve in the punctured disk $0 < |z|<2$.  If we start near $z=1$ with a certain
accuracy $\varepsilon$ and go around the origin and back to $z=1$ again, how much
accuracy will remain?
We will not attempt to give a sharp solution to this problem,
but it is the example that motivated our choice of a strip
of length $2\pi$ for the numbers quoted in the introduction.  
Our estimates suggest
that the number of accurate digits may be reduced by a
factor as great as
$(\pi/4)\exp(\pi^2) \approx \hbox{15,000}$, or 
$\exp(2\kern .4pt \pi e) \approx \hbox{26,000,000}$ for the chain-of-disks method.

\section{Conclusion}

A compelling presentation of the practical side of
analytic continuation can be found in the book to appear
by Fornberg and Piret~\cite{fp}.  In the case of exactly known
functions, although one could use the
chain-of-disks idea in principle, Taylor series play little role 
in practice.  A far more
powerful approach is to find an analytical method to transform one
formula defining a function (a formula being after all a finite
object, unlike an infinite set of Taylor coefficients)
into another formula with a new region of validity.
For the most famous of all examples, the
Dirichlet series for the Riemann zeta function converges
only for $\Real z > 1$, but other representations extend $\zeta(z)$
to the whole complex plane.

The present paper has concerned the case of inexactly known
functions.  Here, for continuation of functions from a disk to a
larger disk, or from an interval to an ellipse, algorithms related
to Taylor or Chebyshev series are effective, as has
been discovered by various authors and we have illustrated by
Figure~\ref{chebfig} from Chebfun.
A third equivalent
context would be analytic continuation of a periodic function into
a strip by Fourier series.
The question is, what can one
do to continue a function beyond the disk/ellipse/strip of convergence
of its Taylor/Chebyshev/Fourier series?
Our theorems show that if all one knows is analyticity and boundedness
along certain channels, then 
accuracy may be lost at a precipitous exponential
rate, better than the chain-of-disks but only by a constant.
However, the assumption of analyticity just
in a channel is more pessimistic than necessary in many applications.
Functions arising in applications rarely have natural boundaries or
other beautiful pathologies of analytic function theory, however
generic such structures may be from a certain abstract point of
view; they are far more likely to be analytic everywhere apart from
certain poles and branch points.  In practice, rational functions
are the crucial tool for analytic continuation in such cases, and
when they work, their convergence is typically exponential, just as
we have found for series-based methods in a disk~\cite{bgm,eiermann,atap}.
It would be an
interesting challenge to develop theorems for meromorphic functions
analogous to what we have established here in the analytic case,
and a discussion with some of this flavor can be found
in~\cite{millerb}.

\begin{acknowledgements}
The early stages of this work benefited from discussions with Marco
Fasondini, Bengt Fornberg, Yuji Nakatsukasa, and Olivier S\`ete.
The first version of the
paper was written during a sabbatical
visit to the Laboratoire de l'Informatique du Parall\'elisme
at ENS Lyon in 2017--18 hosted by Nicolas Brisebarre, Jean-Michel Muller,
and Bruno Salvy.  It was improved in revision by
suggestions from Marco Fasondini, Daan Huybrechs,
Alex Townsend, Marcus Webb, Kuan Xu, and especially Elias Wegert.
\end{acknowledgements}

\end{document}